\documentclass[12pt,reqno]{amsart}
\usepackage[margin=1in]{geometry}
\makeatletter
\def\section{\@startsection{section}{1}%
	\z@{.7\linespacing\@plus\linespacing}{.5\linespacing}%
	{\bfseries%\normalfont\scshape
		\centering
}}
\def\@secnumfont{\bfseries}
\makeatother
\usepackage{amsmath,amssymb,amsthm,graphicx,amsxtra, setspace}
\usepackage[utf8]{inputenc}
\usepackage{charter}
\usepackage{mathrsfs}
\usepackage{alltt}
\usepackage{relsize}
\usepackage{hyperref}
\usepackage{aliascnt}
\usepackage{tikz}
\usepackage{mathtools}
\usepackage{multicol}
\usepackage{upgreek}
\usepackage{graphicx,type1cm,eso-pic,color}
\allowdisplaybreaks
\usepackage{scalerel,stackengine}
\stackMath
\newcommand\reallywidehat[1]{%
	\savestack{\tmpbox}{\stretchto{%
			\scaleto{%
				\scalerel*[\widthof{\ensuremath{#1}}]{\kern-.6pt\bigwedge\kern-.6pt}%
				{\rule[-\textheight/2]{1ex}{\textheight}}%WIDTH-LIMITED BIG WEDGE
			}{\textheight}%
		}{0.5ex}}%
	\stackon[1pt]{#1}{\tmpbox}%
}
\newtheorem{theorem}{Theorem}[section]

\newaliascnt{lemma}{theorem}
\newtheorem{lemma}[lemma]{Lemma}
\aliascntresetthe{lemma} 

\newaliascnt{proposition}{theorem}
\newtheorem{proposition}[proposition]{Proposition}
\aliascntresetthe{proposition}

\newaliascnt{assumption}{theorem}

\aliascntresetthe{assumption}

\newaliascnt{corollary}{theorem}

\aliascntresetthe{corollary}

\newaliascnt{definition}{theorem}
\newtheorem{definition}[definition]{Definition}
\aliascntresetthe{definition}

\newaliascnt{example}{theorem}

\aliascntresetthe{example}

\newaliascnt{remark}{theorem}
\newtheorem{remark}[remark]{Remark}
\aliascntresetthe{remark}

\newaliascnt{hypothesis}{theorem}

\aliascntresetthe{hypothesis}

\newaliascnt{property}{theorem}

\aliascntresetthe{property}

\usepackage[utf8]{inputenc}
\usepackage{amsmath}
\usepackage{enumerate}
\usepackage{bm}
\usepackage[T1]{fontenc}
\usepackage{amsthm}
\usepackage{amsfonts}
\usepackage{amssymb}
\usepackage{graphicx}
\usepackage{enumerate}
\usepackage[english]{babel}
\usepackage{comment}
\newcommand{\Om} {\Omega}
\newcommand{\pa} {\partial}
\newcommand{\be} {\begin{equation}}
	\newcommand{\ee} {\end{equation}}
\newcommand{\bea} {\begin{eqnarray}}
	\newcommand{\eea} {\end{eqnarray}}
\newcommand{\Bea} {\begin{eqnarray*}}
	\newcommand{\Eea} {\end{eqnarray*}}

\newcommand{\de} {\delta}

\newcommand{\De} {\Delta}
\newcommand{\la} {\lambda}

\newcommand{\noi} {\noindent}

\usepackage{amssymb}
\usepackage{amsthm}
\usepackage{xcolor}

\newcommand{\Addresses}{{% additional braces for segregating \footnotesize
		\footnote{
			\noindent \textsuperscript{1}School of Mathematics,
			Indian Institute of Science Education and Research, Trivandrum (IISER-TVM),
			Maruthamala PO, Vithura, Thiruvananthapuram, Kerala, 695 551, INDIA.  \par\nopagebreak \noindent
			\textit{e-mail:} \texttt{dhanya.tr@iisertvm.ac.in}

			\noindent \textsuperscript{2}School of Mathematics, Indian Institute of Science Education and Research, Trivandrum (IISER-TVM),
			Maruthamala PO, Vithura, Thiruvananthapuram, Kerala, 695 551, INDIA  \par\nopagebreak \noindent
			\textit{e-mail:} \texttt{indulekhams17@iisertvm.ac.in}

			\noindent \textsuperscript{2}School of Mathematics, Indian Institute of Science Education and Research, Trivandrum (IISER-TVM),
			Maruthamala PO, Vithura, Thiruvananthapuram, Kerala, 695 551, INDIA  \par\nopagebreak \noindent
			\textit{e-mail:} \texttt{ritabrata20@iisertvm.ac.in}
			
			\noindent \textsuperscript{*}Corresponding author.

			\medskip\noindent
			{\bf Acknowledgments:}   R Dhanya was supported by
	DST/INSPIRE/04/2015/003221 and R Jana was supported by Prime Minister's Research Fellowship when this work was carried out. We thank anonymous referee for their valuable comments which helped in improving the presentation of the paper.
			
}}}

\begin{document}
	
	\title[SCP for a p-Laplace equation involving singularity and its applications]{ Strong comparison principle for a p-Laplace equation involving singularity and its applications	\Addresses	}

	\author[R.Dhanya, M.S. Indulekha and Ritabrata Jana ]
	{R.Dhanya\textsuperscript{1*}, M.S. Indulekha\textsuperscript{2} and Ritabrata Jana\textsuperscript{3}} 
\maketitle

	\begin{abstract}
We prove a strong comparison principle for radially decreasing solutions $u,v\in C_{0}^{1,\alpha}(\overline{B_R})$  of the singular equations $-\Delta_p u-\frac{\lambda}{u^\delta}=f(x)$ and  $-\Delta_p v-\frac{\lambda}{v^\delta}=g(x)$ in $B_R,$ where  $1<p\leq 2 , \; \delta\in (0,1)$ and $\lambda>0.$   We assume that $f $ and $ g$ are continuous radial functions with $0 \leq f \leq g$ and  $f\not \equiv g$ in $B_R.$ Also, a counterexample is provided where the strong comparison principle is violated when $p>2.$ In addition, we prove a three solution theorem for p-Laplace equation as an application of strong comparison principle. This is illustrated with an example.
	\end{abstract}

	

		\section{Introduction}
	\noindent
	Strong comparison principle for p-Laplacian is an inevitable tool in the analysis of partial differential equations. It is useful in
	establishing existence and
	uniqueness results, a-priori estimates, symmetry results,  etc. We consider the following p-Laplace equations for $p\in (1,\infty)$
	\begin{equation}\label{qn}\left.\begin{array}{rll}
			-\De_p u -\frac{\la}{u^\de}&=& f(x) \mbox{ in } B_R \\[2mm]
			-\De_p v -\frac{\la}{v^\de}&=& g(x)\mbox{ in } B_R \\[2mm]
			u=v&=&0 \mbox{ on }\pa B_R
		\end{array}\right\}\end{equation}
	where $B_{R}\subset  \mathbb{R}^n$ is an open ball of radius $R$ centred at origin, $\delta\in (0,1)$ and $\lambda>0.$ The functions $f$ and $g$ belong to $C({B_R})$ and are radial such that $0\leq f\leq g$ and  $f\not\equiv g$ in $B_R.$ We assume that $u$ and $v$ belong to  $C^{1,\alpha}_0(\overline{B_R})$ for some $\alpha\in (0,1).$ Clearly, the solutions $u$ and $v$ of (\ref{qn}) are positive in $B_{R}.$  Given that $f\leq g,$ by weak comparison principle we observe that
	$u\leq v.$ The strong comparison principle (SCP) for $(\ref{qn})$ reads as 
	\begin{equation}\label{eqnscp}
		0<u<v \mbox{ in }B_{R} \,\, \mbox{ and }\,\,\frac{\pa v}{\pa \nu}<\frac{\pa u}{\pa \nu} <0 \mbox{ on }\pa B_{R} 
	\end{equation}
	where $\nu$ denotes the outward normal vector on $\pa B_R.$ The main goal of this article is to investigate to what extend the
	strong comparison principle \eqref{eqnscp} is valid for the p-Laplace equation with a singular nonlinearity as in \eqref{qn}. \\
	
	In the literature, standard methods of strong comparison principle were developed for equations
	\begin{equation}\label{qnQ}\left.\begin{array}{rll}
			-\De_p u -b(x,u)&=& f(x) \mbox{ in }\Om\\
			-\De_p v-b(x,v)&=& g(x)\mbox{ in } \Om\\
			u=v&=&0 \mbox{ on }\pa \Om
		\end{array}\right\}\end{equation}
	where $b(x,\cdot)$ is an increasing function for each $x.$ If $u\le v$ and $f(x) \le g(x),$ then we have  $f^* \le g^*,$ where $f^*:=b(x,u)+f(x)$ and $g^*:=b(x,v)+g(x).$ Now the  comparison principle of \cite{lucia2004strong} is applicable for $-\Delta_p u=f^*$ and $-\Delta_p v=g^*$ and yields $u(x)<v(x)$ for all $x\in \Omega.$ On the other hand, the above technique is no longer applicable for (\ref{qn}) as the function $b(x,\cdot)$ is decreasing for each $x.$ 
	
	Giacomoni et. al. in \cite{giacomoni2007sobolev} derived a  strong comparison principle for quasilinear elliptic equation with singular non-linearity. Here the authors proved that  $u<v$ in $\Om$ for the same set of equations (\ref{qn}), but with a stronger assumption $0\leq f<g$ in $\Omega.$ In contrast to this, we no longer assume $f<g$ and hence the result obtained is stronger. In  \cite{papageorgiou2015bifurcation}, the  SCP is shown  for  PDE of the type $ -\De_p u -\frac{\la}{u^\de}+\sigma u^{p-1} = f(x)$  with similar assumptions as in \cite{giacomoni2007sobolev}. It is noteworthy to mention that in both these articles authors have used the fact that $g-f$ attains a  positive minimum in any compact subset of $\Omega.$ Our main focus  here is to investigate the validity of  SCP relaxing this condition. In section 2, we state the main result as Theorem \ref{MainSCP}, where we prove that the SCP is valid in $B_R$ if $1< p\leq 2$ and $0\leq f\leq g.$ In addition to this, we provide a counterexample for the SCP when $p>2.$ 
	\begin{theorem}\label{MainSCP}
		Let $1<p\leq 2,\; \la>0 $ and $f,g$ be continuous radial functions in $B_R$ such that $0\leq f\leq g$ in $B_R$ and $f\not\equiv g$ in $B_R.$ Assume that  $u,v\in C^{1,\alpha}(\overline{B_R})$, are radially decreasing solutions of  $-\Delta_p u-\frac{\lambda}{u^\delta}=f(x)$ and  $-\Delta_p v-\frac{\lambda}{v^\delta}=g(x)$, $u=v=0$ on $\pa B_{R}$. Then $0<u<v$ in $B_R$ and $\frac{\pa v}{\pa \nu}<\frac{\pa u}{\pa \nu}<0$ on $\pa B_R.$
	\end{theorem}
	\noi If $f$ and $g$ are $L^\infty$ functions in $B_R,$ then  by the regularity results in \cite{giacomoni2007sobolev}, the solutions $u$ and $v$ belong to $C^{1,\alpha}(\overline{B_R})$. If we assume that $f,g$ are radial and radially decreasing, the solutions are expected to be radially decreasing by a recent work of \cite{esposito2020hopf}.
	\par The existence of multiple solutions of elliptic problems is another interesting area of research. In the third section of this paper, we shall see how an SCP is helpful in obtaining a third solution  when two pairs of ordered sub and super-solutions are known(see also the example given at the end of section \ref{sec:3solthm}). In this regard, we consider the following elliptic problem in a bounded open set $\Om$ in $\mathbb{R}^N:$
	\begin{equation}\begin{array}{rll} \label{p2intro}
			-\Delta_p u =\lambda(\frac{1}{u^\delta}+G(u))  & \text{in } \Omega \;\;  ;
			\;\; u > 0  \text{ in } \Omega , \;\; & u=0   \text{ on } \partial \Omega .
	\end{array}\end{equation}
	We assume that $0<\delta<1$  and the function $G: \mathbb{R} \rightarrow [0,\infty) $ is monotonically increasing in $\mathbb{R}^+$ with $G(0)=0.$  We define the solution operator $A_G$ in definition \ref{Agdefi}, section \ref{sec:3solthm} and prove the three solution theorem.  
	
	\begin{theorem}
		\label{3sol}
		(Three solution theorem) Suppose there exists two pairs of ordered sub and supersolutions $(\psi_1,\phi_1)$ and $(\psi_2,\phi_2)$ of (\ref{p2intro})  with the property $\psi_1\leq \psi_2 \leq \phi_1$, $\psi_1\leq \phi_2 \leq \phi_1$ and $\psi_2 \not \leq \phi_2 $. Additionally assume that $\psi_2,\phi_2$ are not solutions of (\ref{p2intro}) and $A_G(\phi_2)<\phi_2$ and $A_G(\psi_2)>\psi_2.$
		Then there exists at least three solutions $u_i,i=1,2,3$ for (\ref{p2intro}) where $u_1\in [\psi_1,\phi_2]$, $u_2\in [\psi_2,\phi_1]$ and $u_3\in [\psi_1,\phi_1] \setminus([\psi_1,\phi_2] \cup [\psi_2,\phi_1])$.
	\end{theorem}
	%\label{}
	
	
	\section{Strong Comparison Principle}
	In this section we prove the main result,\textbf{ Theorem \ref{MainSCP}:}
	\begin{proof}
		Given that $f\leq g,$ using the test function $(u-v)^+$ in the weak formulation of the problem we can find that $u(x)\leq v(x)$ $\, \forall x \in B_R.$ Now, for any $0<r<R$, define $U_{r}:= B_{R}\setminus B_{r}.$ Since $u$ and $v$ are radially decreasing in $B_R\setminus \{0\}$ we have  $\frac{d u}{d r}<0$ and $\frac{d v}{d r}<0.$   Next we write $w=v-u$ and following the idea of \cite{giacomoni2007sobolev} the system of equations in $(\ref{qn})$ can be re-written as
		\begin{equation}\label{qn1}\left.\begin{array}{rll}
				-div(A(x)\nabla w)- \lambda B(x)w & = &g-f \geq 0 \ \text{in}\ U_{r}\\
				w& \geq & 0 \ \text{on}\ \partial U_{r}
			\end{array}\right.
		\end{equation}
		for a matrix $A(x)=[a_{ij}(x)]$ and a scalar function $B(x).$ Here,  
		\begin{eqnarray}\nonumber
			&&a_{ij}(x)=\int_{0}^{1}|(1-t)\nabla u(x)+t\nabla v(x)|^{p-2}\Big[\delta_{ij}+(p-2)\frac{((1-t)u_{x_i}+tv_{x_i})((1-t)u_{x_j}+tv_{x_j})}{|(1-t)\nabla u(x)+t\nabla v(x)|^{2}}\Big]dt\\ \nonumber
			&&\mbox{ and } B(x)=-\delta \int_{0}^{1}\frac{dt}{((1-t)u(x)+tv(x))^{\delta +1}}.
		\end{eqnarray}
		Using the assumptions on $u$ and $v,$ we  note that $A(x)=[a_{ij}(x)]$ is uniformly elliptic in $U_r$ for every $r>0.$  We now fix an $r_0>0$ such that $f-g\not \equiv 0$ in $U_{r_0},$ which is possible as $f,g$ are assumed to be continuous in $B_R.$ Now applying the strong maximum principle Theorem 2.5.2 of \cite{pucci2007maximum} we conclude that $w>0$ in $U_{r}$ for all $r<r_0.$  In fact this implies that $w(x)>0$ for all $x\neq 0.$
		
		In the next step, by exploiting the ideas in section 3 of \cite{cuesta2000strong} we will show that $w$ is strictly positive in $B_{r_0}$ as well  .
		Using the radial symmetry of solutions, the problem (\ref{qn}) can be reduced to a system of ODEs:
		\begin{eqnarray}
			u_{1}^{'}=\alpha(r,u_{2}),\; u_{1}(r_{1})=u_{1,0}\nonumber\\
			u_{2}^{'}=-\frac{N-1}{r}u_{2}+\beta_{f}(r,u_{1}),\; u_{2}(r_{1})=u_{2,0}
		\end{eqnarray}
		where $r_{1} \in (0,R)$, $u_{1}(r)=u(r)$, $u_{2}(r)=|u^{'}(r)|^{p-2}u^{'}(r).$  We denote by $\beta_f(r,y)$ the function $-(\frac{\lambda}{y^{\delta}}+f(r)) $ and
		$\alpha(r,y):(0,R)\times \mathbb{R}\rightarrow \mathbb{R}$ is given by
		\begin{equation} \alpha (r,y)= \left\{
			\begin{array}{@{}rl@{}}
				y^{\frac{1}{p-1}} & \text{if}\; y\geq 0\\
				|y|^{\frac{1}{p-1}} & \text{if}\; y<0.
			\end{array}
			\right.
		\end{equation}
		Clearly $u_{1}(R)=u_{2}(0)=0$. Analogously we can write
		\begin{eqnarray}
			v_{1}^{'}=\alpha(r,v_{2}),\; v_{1}(r_{1})=v_{1,0}\nonumber\\
			v_{2}^{'}=-\frac{N-1}{r}v_{2}+\beta_{g}(r,v_{1}),\; v_{2}(r_{1})=v_{2,0}
		\end{eqnarray}
		where $v_{1}(R)=v_{2}(0)=0$. \\
		Suppose $u(r^{'})=v(r^{'})$ for some $r^{'}< r_0 $(where $r_0$ is as in the first part of the proof). As $w\geq 0$ in $B_{R}$, its minimum is attained at $r'$ and hence $\frac{d w}{d r}(r')=0$. Taking $r_{1}=r'$ in the systems of ODE, $u_{1,0}=v_{1,0}$ and $u_{2,0}=v_{2,0}$. For the function $b(x,u)=\lambda u^{-\delta}$ we have $0\leq -\frac{\pa b}{\pa u}\in L^{\infty}_{loc}((-R,R)\times (0,\infty)),$ and hence by using Lemma 3.2 of \cite{cuesta2000strong}, we obtain $v_{1}(r) \leq u_{1}(r)\; \forall \; r \in [r_{1},R)$ which contradicts the fact that $w>0$ in $U_{r_0}. $ Therefore, $0<u<v$ in $B_{R}.$ Finally we note that since $w>0$ in $B_R,$ we can apply Theorem 2.7.1 of Pucci and Serrin\cite{pucci2007maximum} to conclude that
		$\frac{\pa v}{\pa \nu}<\frac{\pa u}{\pa \nu}<0.$
	\end{proof}
	From a careful observation of the above proof we note that the hypothesis of the Theorem \ref{MainSCP} can be modified as in the next theorem and still the strong comparison principle holds. 
	\begin{theorem}
		Let $1<p\leq 2$ and $u, v$ be positive radially decreasing solutions of (\ref{qn}). Also assume that $f,$ $g$ are continuous radial functions in $B_{R}$ such that $f\leq g$ and $f\not \equiv g.$ Then $u(x)<v(x)$ for all $x\in B_R.$
	\end{theorem}
	When $p>2,$ under the given assumptions of the above theorem we can show that $u<v$ in $B_R\setminus\{0\}.$ On the other hand, when $1<p\leq 2,$ our Theorem \ref{MainSCP} uses the smoothness of the map $t\rightarrow t^{\frac{1}{p-1}}$  along with the Muller Kamke theorem \cite{smith2008monotone} to prove $u(0)<v(0).$ In the next example we prove that the above result(Theorem 2.1) need not be true when $p>2.$\\[2mm]
	\textbf{ Counter example to Theorem 2.1 when $p>2$: }
	For $0< \theta< \infty$, define $u_{\theta}(x):=1-r^{\theta}$ and $f_{\theta}(x):=((p-1)(\theta -1)-1+N)\theta^{(p-1)}r^{(p-1)(\theta-1)-1}-\lambda(1-r^{\theta})^{-\delta}$, where $r=|x|$.
	Clearly,
	\begin{eqnarray}
		-\Delta_{p}u_{\theta}-\lambda u_{\theta}^{-\delta}=f_{\theta}\  \text{in}\  B_{1}\nonumber\\
		u_{\theta}=0\  \text{on}\  \partial B_{1}
	\end{eqnarray}
	Also, $u_{\theta}>0$ in $B_{1}$ and $f_{\theta} \in C(B_{1})$ for all $\theta \in (0, \infty)$. We observe that $u_{\theta}(0)=1$ for all $\theta$ and $u_{\theta_{1}}(x)<u_{\theta_{2}}(x)$  for all $x$ in $B_{1}\setminus\{0\}$ when $0<\theta_{1}<\theta_{2}< \infty$. We claim that, we can choose $\theta_{1},\theta_{2}$ and $\lambda>0$ appropriately so that $f_{\theta_{1}}(x)\leq f_{\theta_{2}}(x)$ in $B_{1}$ and thus the strong comparison principle (Theorem 2.1) is violated. To this end, it is enough to prove that $\partial_{\theta}f_{\theta}\geq 0$. Now,
	\begin{eqnarray}
		\partial_{\theta}f_{\theta}(r)&=&(p-1)(p(\theta -1)+N)\theta^{p-2}r^{(p-1)(\theta-1)-1}+\nonumber \\
		& &(p-1)(p(\theta-1)+N-\theta)\theta^{p-1}r^{(p-1)(\theta-1)-1}ln(r)+\nonumber \\
		& & (-\lambda \delta(1-r^{\theta})^{-\delta-1}r^{\theta}ln(r)) \nonumber
	\end{eqnarray}
	Define $l_{p}(\theta):=\frac{1}{\theta}+\frac{1}{p(\theta-1)+N-\theta}>0$ for $1<\theta<\infty$. If $r\in [e^{-l_{p}(\theta)},  1],$ the first two summands of $\partial_{\theta}f_{\theta}$ give non-negative sum and since $\lambda>0$ the third summand is also positive. This gives $\partial_{\theta}f_{\theta}(r)\geq0$ when $e^{-l_p(\theta)}\leq r\leq 1.$ Next when $r\in (0,e^{-l_{p}(\theta)})$, we first choose $\theta \geq \frac{p}{p-2}$ so that we get
	\begin{eqnarray}\nonumber
		(p-1)(p(\theta-1)+N-\theta)\theta^{p-1}r^{(p-1)(\theta-1)-1-\theta}-\lambda \delta(1-r^{\theta})^{-\delta-1}
		\leq (p-1)(p(\theta-1)+N-\theta)\theta^{p-1}-\lambda\delta.
	\end{eqnarray}
	Now we choose $\lambda$ large enough, for instance $\lambda \delta \geq (p-1)(p(\theta-1)+N-\theta)\theta^{p-1}$ so that the sum of last two terms in $\pa_\theta f_\theta (r)$ is positive.  The first term of $\pa_\theta f_\theta (r)$ is always positive and thus $\pa_\theta f_\theta (r)\geq 0$ for $r\in (0, e^{-l_p(\theta)})$ as well.  Thus we conclude that  the strong comparison principle does not hold true if we choose $\frac{p}{p-2} \leq \theta_1<\theta_2 <\infty$ and $\lambda$ large enough.\hfill\qed

	\section{Three Solution Theorem} \label{sec:3solthm}
	In this section we consider the following quasilinear BVP with singular nonlinearity:
	\begin{equation}\begin{array}{rll} \label{p2}
			-\Delta_p u =\lambda(\frac{1}{u^\delta}+G(u))  & \text{in } \Omega \;\;  ;
			\;\; u > 0  \text{ in } \Omega , \;\; & u=0   \text{ on } \partial \Omega .
	\end{array}\end{equation}
	$\Omega$ is a bounded open subset of  $\mathbb{R}^N$, $N \geq 1$ with smooth boundary $\partial \Omega$ and $0<\delta<1.$ The function $G: \mathbb{R} \rightarrow [0,\infty) $ is monotonically increasing in $\mathbb{R}^+$ with $G(0)=0.$ We prove the existence of three solutions of (\ref{p2}) whenever there exists two pairs of ordered sub and super solutions. We use a technique similar to that in \cite{dhanya2015three}, where the authors have proved this result for the linear case $p=2$. We remark here that all the  results in this section can be concluded for $-\Delta_p u =\lambda(\frac{c}{u^\delta}+G(u))$ where $c$ is any positive constant. 
	\begin{definition}
		A function $u\in C^{1,\alpha}(\Bar{\Omega})$ is said to be a sub-solution(super solution) of (\ref{p2}) if $u>0$ in $\Omega,$ $u=0$ on $\partial \Omega$ and
		
		$$\int_{\Omega}|\nabla u|^{p-2} \nabla u \nabla \phi \leq (\; \geq \;)\lambda\int_{\Omega}(\frac{1}{u^\delta}+G(u)) \phi \mbox{\;\;  } $$
		holds for all non-negative test functions $\phi\in C^{\infty}_c(\Omega)$.
		If a function $u$ is both sub solution and super solution, then it is called a solution of (\ref{p2}).
	\end{definition}
	\begin{definition}
		Given $\lambda>0$ and $0<\delta<1$, we define $\xi_\lambda $ as the unique positive solution of $ -\Delta_p \xi_\lambda= \lambda \xi_\lambda^{-\delta}$ in $\Omega$; $\xi_\lambda|_{\partial\Omega}=0$. By \cite{giacomoni2007sobolev}, we know that there exists positive constants $l, L$ for which $l \, d(x) \leq \xi_\lambda \leq L\, d(x)$, where $d(x)=d(x,\pa \Om)$.
	\end{definition}
	\begin{definition}\label{Agdefi}
		For a given $\la>0,$ we define the map $A_G : C_0(\Bar{\Omega}) \to C_{0}^{1,\alpha}(\Bar{\Omega})$ as $A_G(u)=w$ iff w is a weak solution of $-\Delta_p w -\frac{\lambda}{w^\delta} =\lambda G(u)  \text{ in } \Omega \; ; \; w > 0   \text{ in } \Omega \; , \;
		w=0    \text{ on } \partial \Omega $.
	\end{definition}
	\begin{lemma}
		The map $A_G$ is well defined, monotone operator from $C_0(\Bar{\Omega})$ to $C_0^{1,\alpha}(\Bar{\Omega})$.
	\end{lemma}
	\begin{proof}
		For a given $u,$ existence and uniqueness of a non-negative weak solution $w=A_G(u)\in W^{1,p}_0(\Omega)$ can be proved by the minimization of a suitable energy functional as discussed in Lemma 3.1 of \cite{giacomoni2007sobolev} {or by following the idea of proof of Theorem 3.2 in \cite{GST}.} Again using the results in Appendix B of \cite{giacomoni2007sobolev}, it can be shown that $w\in C^{1,\alpha}(\Bar{\Omega}).$ Now the monotonicity of the map $A_G$ easily follows as $G$ is assumed to be a monotonically increasing function.
	\end{proof}
	
	We define $e \in C^{1,\alpha}(\Bar{\Omega})$ as the unique positive solution of $ -\Delta_p e=1 \ \text{in } \Omega$ with zero Dirichlet boundary condition.  $C_e(\Bar{\Omega})$ is the set of functions in  $C_0(\Bar{\Omega})$ such that $|u|\leq t e(x)$ for some $t>0.$ $C_e(\Bar{\Omega})$ is a Banach space equipped with the norm $\|u\|_e = \inf \{t>0: |u(x)|\leq t e(x)\}$ (see \cite{dhanya2015three} for more details).
%	\textcolor{blue}{In the next proposition we will prove that the map $A_G$ is compact and continuous from $C_e(\Bar{\Omega)}$ to itself. In \cite{GST}, a similar idea is used to prove the compactness of an operator for a coperative system of equations with sigular nonlinearity. Authors have also proved the existence of solution using a fixed point technique.}
	\begin{proposition}
The map $A_G:C_e(\Bar{\Omega}) \longrightarrow C_e(\Bar{\Omega})$ is completely continuous.
	\end{proposition}
	\begin{proof}
		Recalling the continuous embedding $C_0^{1,\alpha}(\Bar{\Omega}) \hookrightarrow C_0^1(\Bar{\Omega}) \hookrightarrow C_e(\Bar{\Omega}) \hookrightarrow C_0(\Bar{\Omega})$ it is enough to show that $A_G:C_0(\Bar{\Omega}) \longrightarrow C_0^{1,\alpha}(\Bar{\Omega})$ is continuous. Let \{$u_h$\} $\subset C_0(\Bar{\Omega})$ be such that $\|u_h-u\|_{C_0(\Bar{\Omega})} \rightarrow 0 $ as $h \rightarrow 0$. Let  $A_G(u)=w$ and $A_G(u_h)=w_h$. Since $G$ is positive, we get  $-\Delta_p w_h -\frac{\lambda}{w_h^\delta} =\lambda G(u_h) \geq 0 $. Using weak comparison principle, we conclude $w_h \geq \xi_{\lambda} \geq c d(x)$. Thus, for some positive constants $C$ and $C_1$ independent of $h,$ we have
		$$w_h^{-\delta} + G(u_h) \leq \frac{C}{d(x)^\delta} \leq  \frac{C_1}{\xi_{\lambda}^{\delta}}.$$  Again from the weak comparison principle we have  $\xi_{\lambda} \leq w_h \leq k \xi_{\lambda}.$ Now we can use Theorem B.1 of \cite{giacomoni2007sobolev} and obtain a $C_0^{1,\alpha}$ uniform bound for $\{w_h\}$, that  is, there exists $M>0$ such that 
		\begin{equation*}
		   \sup_h \|w_h\|_{C_0^{1,\alpha}(\Bar{\Omega})} \leq M. 
		\end{equation*}
		By the compact embedding $C_0^{1,\alpha}(\Bar{\Omega}) \subset \subset C_0^{1,\alpha'}(\Bar{\Omega})$ where $0<\alpha'<\alpha,$ the sequence  $w_h$ has a convergent subsequence in $C_0^{1,\alpha'}(\Bar{\Omega})$, namely $\{w_{h_i}\}$. The uniqueness of weak solution of the equation
		$-\Delta_p w -\frac{\lambda}{w^\beta}=\lambda G(u) $ would imply that $w_{h_i}\rightarrow w$ in  ${C_0^{1,\alpha'}(\Bar{\Omega})}$. Through a standard subsequence argument it can be shown that $w_h\rightarrow w$ in $C^{1,\alpha'}_0(\Bar{\Om})$ and thus the the map $A_G : C_0(\Bar{\Om})\rightarrow C_0^{1,\alpha'}(\Bar{\Omega})$ is continuous.  Once again using Ascoli Arzela theorem it is easy to prove that the map $A_G$ is completely continuous from $C_e(\Bar{\Om})$ to itself.
	\end{proof}
{Authors in \cite{GST} consider a system of quasilinear equations with a singular non-linearity and prove that the  associated operator is completely continuous. Furthermore, they show the  existence of its  solution using Schauder's fixed point theorem. Our aim is to use a fixed point theorem due to Amann \cite{amann1976fixed} to prove the existence of three solutions to \eqref{p2}.}
	\\
$\mathbf{Proof \; of \; Theorem \; \ref{3sol}: }$ Existence of two solutions  $u_1\in [\psi_1,\phi_2]$ and  $u_2\in [\psi_2,\phi_1]$ is straight forward as the map $A_G$ is monotone and completely continuous. The proof of existence of a third solution follows as in the case of Laplacian(see Theorem 3.9 of \cite{dhanya2015three}), but we shall briefly describe the underlying idea here. Using the given condition $A_G(\phi_2)<\phi_2$ we note that $\psi_1\leq u_1 < \phi_2.$  Also,
	\begin{equation} \begin{array}{rll}
			-\Delta_p u_1  - \frac{\lambda}{u_1^\delta} & =& \lambda G(u_1) \mbox{ in } \Omega \\[2mm]
			-\Delta_p \phi_2  - \frac{\lambda}{\phi_2^\delta} & \geq &\lambda G(\phi_2) \mbox{ in } \Omega\\
			u_1=\phi_2 &= &0 \mbox{ on } \partial \Omega.
	\end{array}  \end{equation}
	Since $u_1 < \phi_2$ and $G$ is strictly increasing, using Theorem 2.3 of \cite{giacomoni2007sobolev} (or by Theorem 2.7.1 of Pucci and Serrin\cite{pucci2007maximum}) we have $\frac{\pa u_1}{\pa \nu} >\frac{\pa \phi_2}{\pa \nu}$, or $\phi_2-u_1 \geq c_1 e(x)$ for some positive constant $c_1$.
	Similarly  for some constant $c_2>0$ we can show that  $u_2-\psi_2> c_2 e(x).$ Now the open balls,
	$$B_k =\{z\in C_e(\Bar{\Om})\} :\|z-u_k\|_{e} <c_k\}$$
	for $k=1,2$  lie entirely inside $X_1= [\psi_1,\phi_2]$ and $X_2= [\psi_2,\phi_1]$ respectively. Thus we prove that $X_i$ for $i=1,2$ have non-empty interior and appeal to the fixed point theorem of Amann \cite{amann1976fixed} to conclude the existence of a third solution $u_3\in [\psi_1,\phi_1] \setminus([\psi_1,\phi_2] \cup [\psi_2,\phi_1]).$ \hfill\qed\\[3mm]
	Next we wish to understand under what hypothesis the conditions $A_G(\psi_2)>\psi_2$ and $A_G(\phi_2)<\phi_2$ are valid. Let $\psi_2$ be a sub-solution of $(\ref{p2}),$ then $\psi_2$ is also a  weak solution of the BVP
\begin{equation} 
	\left.
	  \begin{aligned}
	       -\Delta_p \psi_2  - \frac{\lambda}{\psi_2^\delta}  =  \la \Tilde{G}(x)  \mbox{ in } \Omega&  \\
	    \psi_2 = 0 \ \text{on}\ \pa \Om&
	  \end{aligned}
	  \right\}
	\end{equation}
% 	\begin{eqnarray}
% 	    -\Delta_p \phi_2  - \frac{\lambda}{\phi_2^\delta} & = & \la \Tilde{G}(x)\, \geq \,   \la G(\phi_2) \mbox{ in } \Omega \nonumber \\
% 	    \phi_2 = 0 \ \text{on}\ \pa \Om
% 	\end{eqnarray}
	for some function $\tilde{G}$ defined on $\Omega$. Since $\psi_2$ is a subsolution we have $\tilde{G}(x)\leq G \circ \psi_2(x).$ In most of the applications, when $\psi_2$ is known the function $\tilde{G}$ happens to be continuous in $\Omega.$ 
	%Let $\Bar{\phi_2}$ denote $A_{g}(\phi_{2})$.
	\begin{comment}
	\begin{equation}\label{seteqn}
		\begin{array}{rllrll}
			-\Delta_p \phi_2  - \frac{\lambda}{\phi_2^\delta} & = & \la \Tilde{g}(x)\, \geq \,   \la g(\phi_2) \mbox{ in } \Omega %-\Delta_p \psi_2  - \frac{\lambda}{\psi_2^\delta} & = & \la \underline{g}(x)\, \leq \,  \la  g(\psi_2) \mbox{ in } \Omega\\[2mm]
			
			\\[3mm] -\Delta_p \Bar{\phi_2}  - \frac{\lambda}{\Bar{\phi_2}^\delta} & = & \la  g(\phi_2) \mbox{ in } \Omega %-\Delta_p \psi_2 ' - \frac{\lambda}{\psi_2'^\delta} & = & \la  g(\psi_2) \mbox{ in } \Omega\\[3mm]
			\\[3mm] \phi_2=\Bar{\phi_2}&=&0 \mbox{ on } \pa\Omega \;\;\;\;\;\; & %\psi_2=\psi_2'&=&0 \mbox{ on } \pa\Omega
		\end{array}
	\end{equation}
	where 	$\Bar{\phi_2}=A_g(\phi_2)$.
	\end{comment} 
	If we write $\omega=A_G(\psi_2),$ then 
	\begin{equation} 
	\left.
	  \begin{aligned}
	       -\Delta_p \omega  - \frac{\lambda}{\omega^\delta}  =  \la G(\psi_2)  \mbox{ in } \Omega&  \\
	    \omega = 0 \ \text{on}\ \pa \Om&
	  \end{aligned}
	  \right\}
	\end{equation}
	If $\Omega= B_R$ is a ball and $1<p\leq 2,$ Theorem 1.1 provides a sufficient condition that ensures the hypothesis of the three solution theorem. This result is stated as the following proposition. Similar results hold true for $\phi_2$ as well.
	\begin{proposition}\label{scp2}
		Let $1<p\leq 2$ and $\Omega=B_R.$ Suppose $\psi_2$ and $A_G({\psi_2})$ are radially decreasing functions in $B_R$.  If $\Tilde{G}(x)$ is a continuous radial function in $B_R,$ then we have $A_G({\psi_2})>\psi_{2}$.
	\end{proposition}
	Next we shall state another condition which can be useful to prove the three solution theorem.
	\begin{proposition}
		\label{scp3}
		Let $1<p<\infty$ and $\Omega$ is an arbitrary bounded open set with smooth boundary.  Assume that $\Tilde{G}(x)< G(\psi_2(x))$ for all $x\in \Omega.$  Then $A_G(\psi_2)>\psi_2.$
	\end{proposition}
	\begin{proof}
		Proof is a straightforward application of the strong comparison principle in Theorem 2.3 of \cite{giacomoni2007sobolev} or Proposition 4 of \cite{papageorgiou2015bifurcation}.
	\end{proof}

%	\begin{remark}\label{conditionG}
%	If there exists a constant $k>0$ such that $G(t)+ kt$ is increasing in $[0,\infty)$, then the assumption $G$ is monotonically increasing can also be relaxed while using Proposition \ref{scp3}.
%	\end{remark}
	
% 	\begin{remark}\label{conditionG}
% 	     Observe that, same results can be concluded for $-\Delta_p u =\lambda(\frac{c}{u^\delta}+G(u))$ where $c$ is any positive constant. %The assumption $G$ is monotonically increasing can also be relaxed if there exists a constant $k>0$ such that $G(t)+ kt$ is increasing in $[0,\infty).$
% 	\end{remark}

	We demonstrate the three solution theorem for an elliptic equation with singularity through an example. %In \cite{ko2011multiplicity}, authors have constructed two pairs of sub and super-solutions for the BVP in the following example and have provided two solutions. This work is extended to the p-q-Laplacian operator in a ball by Acharya et al.\cite{acharya2021existence}. One can define the map $A_G$ analogously for the p-q-Laplacian operator, and the continuity of $A_G$ follows using the estimation in \cite{giacomoni2021sobolev}. A strong comparison principle for the p-q-Laplacian operator is proved in Proposition 7 of \cite{papageorgiou2020nonlinear}. It provides the Proposition 3.3 of our paper. In turn, one can use it to establish a three solution theorem. Arora \cite{arora2021multiplicity} has proven a result in this direction with the p-q-Laplacian operator. We shall prove the existence of a third solution for the problem \eqref{kopr}, proposed at \cite{ko2011multiplicity}, although it is an easy consequence of Arora's work \cite{arora2021multiplicity}.\\
\vspace{0.1 cm}\\
\underline{\textbf{Example 3.1}} \label{example} Ko et al. \cite{ko2011multiplicity} have considered the boundary value problem
	\begin{equation} \label{kopr}
	\left.
	  \begin{aligned}
	       -\Delta_p u &= \lambda \frac{F(u)}{u^\delta} \ \text{in} \  \Om
	       \\
	       u &=0 \  \text{on} \ \partial\Om 
	       \\
	       u & > 0 \ \text{in} \ \Om
	  \end{aligned}
	  \right\}
	\end{equation}
	where $1<p<\infty$, $\delta \in (0,1)$, $\lambda$ is a positive parameter and $\Omega$ is a bounded domain in $\mathbb{R}^N,N \geq1$, with smooth boundary.  It is also assumed that $F \in C^1([0,\infty))$ is a non-decreasing function with $F(u)>0$ for all $u \geq 0$  and $\lim_{u \rightarrow \infty}\frac{F(u)}{u^{\delta+p-1}}=0$. With a few more technical assumptions on $F$, in  \cite{ko2011multiplicity},  authors have established the existence of two positive solutions $u_1, u_2$ of (\ref{kopr}) by constructing two pairs of sub-super solutions $(\psi_1,\phi_1)$and $(\psi_2,\phi_2)$  whenever $\la \in (\lambda_{*}, \lambda^{*})$ as given in the Theorem 1.3 of \cite{ko2011multiplicity}. A model problem was given by $F(u)=e^\frac{\alpha u}{\alpha+u}$ for $\alpha>>1$. We urge the readers to go through the cited reference to know about the exact definition of $\lambda_*, \lambda^*$ and sub-supersolutions  $\psi_i$ and $ \phi_i.$ 
	\par
	
% 	We intend to prove the existence of a third solution to (\ref{kopr}) when $\lambda\in (\la_*,\la^*).$ First we re-write the equation (\ref{kopr}) as (\ref{p2intro}) by considering $G(u)= \frac{F(u)-F(0)}{u^\delta}.$
We intend to modify the construction of the subsolution $\psi_2$ given in \cite{ko2011multiplicity} and by abuse of notation we call the new subsolution also as $\psi_2.$ 
  This reconstruction of $\psi_2$ is necessary to use Proposition \ref{scp3} and we conclude the example by showing \eqref{kopr} has a third solution $u_3$ when $\lambda\in (\la_*,\la^*).$ For ease of notation let us assume that $F(0)=1.$ Now we re-write the equation (\ref{kopr}) as (\ref{p2}) by taking $G(u):= \frac{F(u)-F(0)}{u^\delta}.$ Clearly, Theorem 1.2 would guarantee the existence of a third solution $u_3 \in [\psi_1,\phi_1] \setminus([\psi_1,\phi_2] \cup [\psi_2,\phi_1])$  of \eqref{kopr} if $A_G(\phi_2)<\phi_2$ and $A_G(\psi_2)>\psi_2.$  For our purpose of establishing $A_G(\psi_2)>\psi_2,$ as mentioned before we slightly modify the construction of $\psi_2$ given in \cite{ko2011multiplicity}. We first fix a $\la\in (\la_*,\la^*)$ and define $H(u):=\frac{h(u)}{1+\epsilon_\lambda}$ where $h(u)$ is given in Page no-7 of \cite{ko2011multiplicity}. Here $\epsilon_\lambda$ is chosen in such a way that $\frac{\lambda}{1+\epsilon_\lambda} $ still lies within the interval $(\lambda_{*}, \lambda^{*})$. We now follow the construction of sub-solution $\psi_2$ in \cite{ko2011multiplicity} except for equation number (5) in page 8. If we modify this particular equation (5) in \cite{ko2011multiplicity}, with $-\Delta_p u= \lambda H(u)\; \text{in} \;\Omega;\; u|_{\partial\Omega}=0$ and redo the calculations then the resulting sub-solution verifies the strict inequality $A_G(\psi_2)>\psi_2.$ From the definition of $\phi_2$ in \cite{ko2011multiplicity}, clearly $A_G(\phi_2)<\phi_2.$
	
	%by using $H(u):= \frac{h(u)}{1+\epsilon_\la}$ instead of $h(u)$ where $h$ is as given in page number 7 of \cite{ko2011multiplicity}. We choose $\epsilon_\la$ in such a way that $\frac{\la}{1+\epsilon_\la} \in (\la_*, \la^*)$. Rest of the proof remains same and then the resulting subsolution $\psi_2$ also satisfies the strict inequality $A_g(\psi_2)>\psi_2$. From the definition of $\phi_2$ in \cite{ko2011multiplicity}, clearly $A_G(\phi_2)<\phi_2.$ This establishes the existence of a third solution $u_3 \in [\psi_1,\phi_1] \setminus([\psi_1,\phi_2] \cup [\psi_2,\phi_1])$  of \eqref{kopr}. 

 \par We summarize the above discussion in the following remark.
	\begin{remark}
 		\label{remarkshivaji}
 		The boundary value problem (\ref{kopr}) admits three solutions whenever $\la\in (\la_*,\la^*).$ \qedsymbol
 \end{remark}
 Towards the completion of our work we came across two recent manuscripts  \cite{acharya2021existence} and \cite{arora2022multiplicity} where a problem similar to Example 3.1 is considered for $p$-$q$ Laplacian. In both papers, the authors focus on the construction of two pairs of sub-super solutions either in a ball or in a general domain. Arora \cite{arora2022multiplicity} also establishes a three solution theorem for $p$-$q$ Laplacian with the help of the strong comparison principle given in Proposition 6 of \cite{papageorgiou2020nonlinear}. Though Example 3.1 can be treated as a special case of the work of Arora, we wish to conclude our paper emphasizing that in the light of the strong comparison principle in a ball (Theorem 1.1), our three solution theorem is applicable for more general elliptic boundary value problems. \\


\medskip\noindent

	
	
 \begin{thebibliography}{99}
  
\bibitem{acharya2021existence} Ananta Acharya, Ujjal Das, and Ratnasingham Shivaji. Existence and multiplicity results for
pq-laplacian boundary value problems. Electronic Journal of Differential Equations, Special
Isuue 01(2021), pp. 293-300, 2021.
\bibitem{amann1976fixed}Herbert Amann. Fixed point equations and nonlinear eigenvalue problems in ordered banach
spaces. SIAM review, 18(4):620–709, 1976.
\bibitem{arora2022multiplicity} Rakesh Arora. Multiplicity results for nonhomogeneous elliptic equations with singular nonlin-
earities. arXiv preprint arXiv:2109.03274, 2021.
\bibitem{cuesta2000strong} Mabel Cuesta and Peter Takáč. A strong comparison principle for positive solutions of degen-
erate elliptic equations. Differential and Integral Equations, 13(4-6):721–746, 2000.
\bibitem{dhanya2015three} R Dhanya, Eunkyung Ko, and Ratnasingham Shivaji. A three solution theorem for singular
nonlinear elliptic boundary value problems. Journal of Mathematical Analysis and Applications,
424(1):598–612, 2015.
\bibitem{esposito2020hopf} Francesco Esposito and Berardino Sciunzi. On the höpf boundary lemma for quasilinear problems
involving singular nonlinearities and applications. Journal of Functional Analysis, 278(4):108346,
2020.
\bibitem{giacomoni2007sobolev} Jacques Giacomoni, Ian Schindler, and Peter Takáč. Sobolev versus hölder local minimizers and
existence of multiple solutions for a singular quasilinear equation. Annali della Scuola Normale
Superiore di Pisa-Classe di Scienze, 6(1):117–158, 2007.
\bibitem{GST} Jacques Giacomoni, Ian Schindler, and Peter Takáč. Singular quasilinear elliptic systems and
hölder regularity. Advances in Differential Equations, 20(3/4):259–298, 2015
\bibitem{ko2011multiplicity} Eunkyung Ko, Eun Kyoung Lee, and Ratnasingham Shivaji. Multiplicity results for classes of
infinite positone problems. Zeitschrift für Analysis und ihre Anwendungen, 30(3):305–318, 2011.
\bibitem{lucia2004strong} M Lucia and S Prashanth. Strong comparison principle for solutions of quasilinear equations.
Proceedings of the American Mathematical Society, 132(4):1005–1011, 2004.
\bibitem{papageorgiou2020nonlinear}Nikolaos S Papageorgiou, Vicenţiu D Rădulescu, and Dušan D Repovš. Nonlinear nonhomoge-
neous singular problems. Calculus of Variations and Partial Differential Equations, 59(1):1–31,
2020.
\bibitem{papageorgiou2015bifurcation} Nikolaos S Papageorgiou and George Smyrlis. A bifurcation-type theorem for singular nonlinear
elliptic equations. Methods and Applications of Analysis, 22(2):147–170, 2015.
\bibitem{pucci2007maximum} Patrizia Pucci and James B Serrin. The maximum principle, volume 73. Springer Science and
Business Media, 2007.
\bibitem{smith2008monotone} Hal L Smith. Monotone dynamical systems: an introduction to the theory of competitive and
cooperative systems: an introduction to the theory of competitive and cooperative systems. Num-
ber 41. American Mathematical Soc., 2008.
8

\end{thebibliography}
	 \end{document}